\documentclass[10pt]{article}   

\usepackage{amsmath, amssymb}    
\usepackage{amsthm}
\usepackage{mathrsfs}

\usepackage[utf8]{inputenc}		

\usepackage{enumitem}

\usepackage{xcolor} 

\usepackage{hyperref}
\hypersetup{colorlinks=true,linkcolor=blue,citecolor=blue}     

\usepackage{hyperref}
\hypersetup{colorlinks=true,linkcolor=red,
citecolor=blue} 

\usepackage[normalem]{ulem}

 \title{\textbf{Gaussian fluctuations for the directed polymer partition function for $d\geq 3$ and in the whole $L^2$-region.}}
   
\newcommand{\kU}{\mathscr{U}}
\newcommand{\sZ}{{\bf{\mathcal{W}}}}
 \newcommand{\IP}{{\mathbb P}}
\newcommand{\DP}{{\mathrm P}}
\newcommand{\IE}{{\mathbb E}}
\newcommand{\DE}{{\mathrm E}}

\newcommand{\cvlaw}{\stackrel{{ (d)}}{\longrightarrow}}
\newcommand{\cvIP}{\stackrel{{ \IP}}{\longrightarrow}}
\newcommand{\eqlaw}{\stackrel{\rm{(d)}}{=}}
\newcommand*\cvLone{\overset{L^1}{\longrightarrow}}

\newcommand{\rmd}{\mathrm{d}}
\newcommand{\dd}{\mathrm{d}}

\newcommand{\N}{{\mathbb N}}

 \newcommand{\Z}{{\mathbb Z}}

 \newcommand{\eps}{{\epsilon}}
 \newcommand{\e}{{\varepsilon}}
 
 \newcommand{\R}{{\mathbb R}}

\newtheorem{proposition}{Proposition}[section]
\newtheorem{theorem}{Theorem}[section]
\newtheorem{lemma}{Lemma}[section]
\newtheorem{corollary}[theorem]{Corollary}

\newtheorem{definition}{Definition}[section]

\newtheorem{remark}{Remark}[section]

\makeatletter
\def\blfootnote{\xdef\@thefnmark{}\@footnotetext}
\makeatother

\author{Cl\'ement Cosco\footnote{Department of Mathematics, Weizmann Institute of Science.
\texttt{clement.cosco@weizmann.ac.il}}
\and
Shuta Nakajima
   \footnote{Department of Mathematics and Computer Science of the University of Basel. \texttt{shuta.nakajima@unibas.ch}}
}

\begin{document}

\maketitle

\date{}
\begin{abstract}
We consider the discrete directed polymer model with i.i.d.\ environment and we study the fluctuations of the tail $n^{(d-2)/4}(W_\infty - W_n)$ of the normalized partition function. It was proven by Comets and Liu \cite{CL17}, that for sufficiently high temperature, the fluctuations converge in distribution towards the product of the limiting partition function and an independent Gaussian random variable. We extend the result to the whole $L^2$-region, which is predicted to be the maximal high-temperature  region where the Gaussian fluctuations should occur under the considered scaling. To do so, we manage to avoid the heavy 4th-moment computation and instead rely on the local limit theorem for polymers \cite{S95,V06} and homogenization.
\end{abstract}


\section{Introduction}
\subsection{The model}
The directed polymer model was first introduced by Huse and Henly in the physics literature \cite{HH85} and was reformulated in mathematics by Imbrie and Spencer \cite{IS88}. The model is a description of a long chain of monomers, called a \emph{polymer}, which interacts with impurities that it may encounter on its path. The reader is referred to \cite{C17} for a recent review of the model.  In the discrete case, the model is defined as follows. 

The impurities, also called the \emph{environment}, are modeled by a collection of non-constant, i.i.d.\  random variables $\omega(i,x)$, $i\in \mathbb{N}, x\in \mathbb{Z}^d$, defined under a probability measure $\IP$ of expectation denoted by $\IE$. We moreover assume that:
\begin{equation}
\lambda(\beta) := \log\IE\left[e^{\beta \omega(i,x)}\right] <\infty, \quad \forall \beta \in\mathbb R.
\end{equation}

Let $\Omega=\{(S_{k})_{k\geq 0}, S_k \in \mathbb{Z}^d\}$ be the state space of the trajectories, and $\DP_x$ the probability measure on $\Omega$, such that the canonical process $(S_k)_{k\geq 0}$ is the simple random walk on $\mathbb{Z}^d$ starting at position $x$, i.e.\ under $\DP_x$, $S_1-S_0,\dots,S_{k+1}-S_{k}$ are independent and \[\DP_x(S_0=x) = 1,\quad \DP_x(S_{k+1}-S_k = \pm \mathbf{e}_i) = \frac{1}{2d},\]
where $ \mathbf{e}_i$ is any vector of the canonical basis of $\mathbb{R}^d$. We denote by $\DE_x$ the expectation under $\DP_x$, and $\DP=\DP_0,\DE=\DE_0$. 

Then, the \emph{Gibbs measure of the polymer} ${ \DP_{\beta,n}}$ on $\Omega$ is defined as:
\[\rmd { \DP_{\beta,n}}(S) = \frac{\exp\left\{\sum_{i=1}^n \beta \omega(i,S_i)\right\}}{Z_n(\beta)} \rmd { \DP}(S),\]
where $\beta{\geq 0}$ stands for the \emph{inverse temperature} of the polymer, and where $Z_n(\beta) = \DE\left[\exp\{\sum_{i=1}^n \beta \omega(i,S_i)\}\right]$ is called the \emph{partition function}.

A \emph{polymer path of horizon} $n$ is the realization of $(S_k)_{0\leq k \leq n}$ under the polymer measure ${ \DP_{\beta,n}}$. The parameter $\beta$ models the strength of the interaction of the polymer with the environment: the higher $\beta$, the more the polymer path is tempted to go through high values of the environment.

The normalized partition function:
\begin{equation} \label{eq:defWnek}
W_n = Z_n e^{-n\lambda(\beta)} = \DE\left[ e_n\right], \quad \text{with: } e_n = e^{\beta\sum_{i=1}^n\omega(i,S_i)-n\lambda(\beta)},
\end{equation}
is a mean $1$, positive martingale with respect to the filtration $\mathcal{F}_n$ generated by the variables $\omega(i,x)$, $i\leq n, x\in \mathbb{Z}^d$. The martingale verifies the following dichotomy \cite{CSY04}: for $d\geq 3$ (which will be assumed from now), there exists {a critical parameter $\beta_c=\beta_c(d) \in (0,\infty]$}, such that 
\begin{itemize}
\item For all $\beta<\beta_c$, $W_n\to W_\infty$ a.s., with $\IP(W_\infty > 0) =1$,
\item For all $\beta > \beta_c$, $W_n \to 0$ a.s.
\end{itemize}

The region $(0,\beta_c)$ is called the \emph{weak disorder regime}, while the region $(\beta_c,\infty)$ is called the \emph{strong disorder regime}.
In the weak disorder region, the polymer path is \emph{diffusive} (it was first proved in a more restrained region in \cite{B89,IS88}, then in the whole weak disorder region in \cite{CY06}), while in the strong disorder regime, it is believed that the polymer path should be \emph{superdiffusive}. Moreover, it was shown that for large enough $\beta$, the polymer path localizes \cite{CSY03,BC20,C18}.

The subregion of the weak disorder, where $W_n\to W_\infty$ in $L^2$ is called the $L^2$\emph{-region}. It corresponds to the $\beta$-region (see e.g.\  (9)-(11) in \cite{CL17}):
\begin{equation*}\label{eq:L2condition}
\mathbf{(L2)}\quad\lambda_2(\beta):=\lambda(2\beta)-2\lambda(\beta) < \log (1/\pi_d),
\end{equation*}
where $\pi_d\in(0,1)$ is the probability of return to $0$ of the simple random walk:
\[\pi_d = \DP(\exists n\geq 1, S_n = 0).\]

Moreover, since $\pi_{d+1}<\pi_d$ for all $d\geq 3$ \cite[Lemma 1]{OS96} and $\pi_3=0.3405\dots$ \cite[page 103]{S76}, condition $\mathbf{(L2)}$ is always verified for $\beta$ small enough. As the function $\lambda_2$ is non-decreasing on $\mathbb{R}_+$ and non-increasing on $\mathbb{R}_-$, this implies that
\begin{equation*}
\mathbf{(L2)}\Leftrightarrow \beta\in(0,\beta_2), \text{ with } {\beta_2=\beta_2(d) \in(0,\infty]}.
\end{equation*} 
Then, again by \cite{B89},
\begin{equation} \label{eq:SndMomentWinfty}
\IE\big[W_\infty^2\big] = \begin{cases} \frac{(1-\pi_d)e^{\lambda_2(\beta)}}{1-\pi_d e^{\lambda_2(\beta)}} & \text{if } \beta\in(0,\beta_2),\\
\infty & \text{else}.
\end{cases}
\end{equation}

{It is in fact known that the $L^2$-region is a \emph{strict} sub-region of the weak disorder, i.e.\ that $\beta_{2} < \beta_c$. This has been proved to hold for all $d\geq 3$ \cite{BGH11,BT10,BS10,BS11} (see in particular \cite[Section 1.4]{BS10} in the case where $d=3,4$).}

In the following, we will always assume that $d\geq 3$ and $\beta\in(0,\beta_2)$.


\subsection{The results}
We introduce two additional types of convergences, referring to \cite{CL17}.
\begin{definition}
  Let $Y_n$ be a family of random variables defined on a common probability space $(\Omega,\mathcal{F},\IP)$. Suppose that $Y_n$ converges to some random variable $Y$ in distribution.
  \begin{itemize}
    \item We say that this convergence is stable if for any $B\in\mathcal{F}$ with $\IP(B)>0$, the law of $Y_n$ under the condition $B$ converges to some probability distribution, which might depend on $B$.
    \item We say that this convergence is mixing if it is stable and the limit of conditional law is independent of given $B$. Then this conditional limit is the law of $Y$ itself.
    \end{itemize}
  \end{definition}
\begin{theorem} \label{th:mainTheorem}
For all $\beta\in(0,\beta_2)$, as $n\to\infty$,
\begin{equation} \label{eq:stableMainTh}
n^{\frac{d-2}{4}}(W_\infty - W_n) \cvlaw \sigma W_\infty G,
\end{equation}
and 
\begin{equation} \label{eq:mixingMainTh}
n^{\frac{d-2}{4}}\frac{W_\infty - W_n}{W_n} \cvlaw \sigma G,
\end{equation}
where $G$ is a standard centered Gaussian random variable which is independent of $W_\infty$, and $\sigma=\sigma(\beta)$ is defined in \eqref{eq:defSigma}. Moreover, convergence \eqref{eq:stableMainTh} is stable and convergence \eqref{eq:mixingMainTh} is mixing.
\end{theorem}
\begin{remark}
The parameter $\sigma(\beta)$ blows up at $\beta=\beta_{2}$. We believe that different scaling factors, or other limiting laws should be considered for \eqref{eq:stableMainTh}-\eqref{eq:logMainCV} outside of the $L^2$-region.
\end{remark}

\begin{corollary} \label{cor:mainCor}
For all $\beta\in(0,\beta_2)$, as $n\to\infty$,
\begin{equation} \label{eq:logMainCV}
n^{\frac{d-2}{4}}(\log W_\infty - \log W_n) \cvlaw \sigma G,
\end{equation}
where $G$ and $\sigma$ are as above. Moreover, this convergence is mixing.
\end{corollary}
The proof of Theorem \ref{eq:mixingMainTh} is given in Section \ref{subsec:CLT} and the proof of its corollary can be found in Section \ref{subsec:proofOfMainCor}.

\subsection{Comments and comparaison to branching models}
Our work is an extension of the results of Comets and Liu \cite{CL17} to the whole $L^2$-region. {Although our proof of Theorem \ref{th:mainTheorem} builds on the same central limit theorem for martingales, the authors in \cite{CL17} control the bracket of the martingale through some fourth moment computations that do not carry to the full $L^2$-region, as they require $(W_n)$ to be bounded in $L^4$ which does not hold anymore as $\beta$ gets close to $\beta_2$.
Therefore, new arguments have to be brought in order to avoid fourth moment computations (which are moreover quite heavy in terms of computations.)}. Instead, we make a natural use of the local limit theorem for polymers \cite{S95,V06} and appeal to homogenization via a truncation method. 

 {It seems plausible that \eqref{eq:mixingMainTh} or \eqref{eq:logMainCV} could be proven using the 4th-moment theorem \cite{NP05,NPR10} in a similar way than in \cite{CSZ17b,LZ20}, but we believe that this would require the use of uniform bounds on negative moments of $(W_n)$ which, to our knowledge, have only been shown to hold when imposing extra-conditions on the environment (namely some concentration properties, see \cite{LZ20,CSZ18b,CH02}) that we do not need to assume in this paper. We do not  believe that the fourth moment theorem (or some adaptation of it) would help proving \eqref{eq:stableMainTh} directly because of the presence of $W_\infty$ on the RHS, which in particular has fourth moment that blows up when $\beta$ approaches $\beta_2$.} 

In branching process literature, the study of the rate of convergence and the nature of the fluctuations of the tail of characteristic martingales is a common subject. For the Galton-Watson process, this goes back to \cite{H70,H71} and the rate of convergence is there exponential, while the rate is polynomial in our case. {In the branching random walk framework, recent works have focused on the fluctuations of the tail of Biggins' martingale $Y_n(\theta)$ \cite{IKM18,IK16,VRT02,MP18}. The nature of the fluctuations in the particular case of binary random walk with Gaussian increments is described in \cite[Example 2.1]{IKM18}. Let $\theta_c$ be the critical parameter that separates the regimes where $Y_\infty(\theta)=0$ or $Y_\infty(\theta)>0$ and $\theta_2$ the parameter that separates when $(Y_n(\theta))_n$ is bounded in $L^2$ or not (these are the anologues of $\beta_c$ and $\beta_2$ for the polymer). These parameters satisfy $\frac {\theta_c} 2 < \theta_2 < \theta_c$ and the following happens:
}

{(i) For $\theta\in(0,\frac{\theta_c}{2})$, there is some exponent $r=r(\theta)>1$ such that 
\begin{equation} \label{eq:biggins}
r^n(Y_n(\theta) - Y_\infty(\theta))\cvlaw {c(\theta)} \sqrt{Y_\infty(2\theta)}\, G, \quad n\to\infty,
\end{equation}
where $G$ is a centered Gaussian independent of $Y_\infty$ and $c(\theta)$ is a constant that blows up at $\theta_2$. The result can be compared to \eqref{eq:stableMainTh}, but there are dissimilarities. On the one hand,  the variance $Y_\infty(2\theta)$ vanishes at $\frac{\theta_c}{2}$ and the nature of the fluctuations will change after this point (see (ii) and (iii) below). Note that this switch of behavior happens strictly before the $L^2$-critical parameter $\theta_2$ where $c(\theta)$ blows up. On the other hand, the random variance $W_\infty(\beta)^2$ in \eqref{eq:stableMainTh} remains non-degenerate up until $\beta_c$ and the behavior of the fluctuations changes at $\beta_2$ when $\sigma(\beta)$ blows up. Moreover, in the high-temperature regions, the speed of convergence is exponential (depending on $\theta$) for BRW and polynomial (with constant exponent) for the polymer.}

{(ii) For $\theta = \frac{\theta_c}{2}$, the scaling factor in \eqref{eq:biggins} has to be multiplied by an extra $n^{1/4}$ (this accounts for the fact that $Y_\infty(\theta_c)=0$ in the RHS of \eqref{eq:biggins}) and the limiting law is still Gaussian, but with random variance given by the \emph{derivative martingale} $D_\infty$, which is defined by the a.s.\ limit $\sqrt n W_n(\theta_c) \to D_\infty$ as $n\to\infty$.
}

{(iii) For $\theta\in (\frac {\theta_c} 2,\theta_c)$, the scaling changes and the limit is neither gaussian nor $\alpha$-stable, but it is rather related to the point process description of the extremal points of BRW, see \cite[Example 2.1]{IKM18} for more details.}

{(iv) At $\theta= \theta_c$, fluctuations with respect to the derivative martingale for branching Brownian motion have been shown in \cite{MP18} to be of $1$-stable type under suitable scaling.}
{
\subsection{Relation to the KPZ equation and stochastic heat equation}}
Several recent papers \cite{MSZ16,GRZ18,MU18,DGRZ20,DGRZ18b,CCM20,CCM19}, have focused on the study of the behavior of the regularized SHE (stochastic heat equation) and KPZ (Kardar-Parisi-Zhang) equation in dimension $d\geq 3$. The KPZ equation is formally defined as
\begin{equation} \label{eq:formalKPZ}
\frac{\partial}{\partial t} h = \frac12 \Delta h +  \frac12   |\nabla h |^2+ \beta  \xi,
\end{equation}
for $t\geq 0$ and $x\in\mathbb R^d$,
whith $\xi$ a space-time white noise on $\mathbb R_+ \times \mathbb R ^d$.
For all $d\geq 1$, it is difficult to define what is a solution to \eqref{eq:formalKPZ} because the derivative $\nabla h$ fails to be a function and so one has to make sense of the square of a distribution. In dimension $d=1$, much work has been concecrated to define this equation correctly, with notable contributions from \cite{BC95,H13}.  In dimension $d\geq 3$ however, the SPDE falls into the \emph{super-critical} dimensions type and the problem is not solved by the recent theories \cite{H14,GIP15,KM17,GP18}. 

The classical starting point in studying \eqref{eq:formalKPZ} is to consider the regularized equation
\begin{equation}\label{eq:KPZe}
 \frac{\partial}{\partial t} h_{\e} = \frac12 \Delta h_{\e} +  \bigg[\frac 1 2   |\nabla h_\e |^2  - C_\e\bigg]+ \beta \e^{\frac{d-2}2}   \xi_{\e} \;,\quad\,\,   h_{\e}(0,x) = h_0(x),
\end{equation}
where
$\xi_{\e}(t,x) = \int \phi_\e(x - y) \xi(t,y) \rmd y,$ is a mollified white noise with $\phi_\e = \e^{-d} \phi(\e^{-1} x)$ and $\phi$ being a smooth, compactly supported, symmetric function on $\mathbb R ^d$.
A closely related object is the function $u_\e$ defined via the so-called Hopf-Cole transform $h_\e(t,x) = \log u_\e(t,x)$, which solves the regularized \emph{stochastic heat equation}
\begin{equation}\label{eq:SHEintro}
\frac{\partial}{\partial t} u_{\e} = \frac12 \Delta u_{\e} + \beta \e^{\frac{d-2}2} u_{\e} \, \xi_{\e}, \quad u_\e(0,x) = u_0(x). 
\end{equation} 
Then, one tries to study the asymptotics of $h_\e$ and $u_\e$ when the mollification is removed ($\e\to 0$). 

It turns out that via the Feynman-Kac formula, one can check that for $u_0\equiv 1$, the solution of the regularized SHE satisfies (for fixed $t$)
$u_\e(t,x) \eqlaw \sZ_{\e^{-2}t} \left(\beta, \e^{-1} x\right)$, where $\sZ_{T}(\beta,x)$ is the normalized partition function of a continuous directed polymer model with Brownian path started at $x\in\mathbb R^d$ and white noise environment, see \cite{MSZ16}. Note that we can alternatively associate the discrete log-polymer partition function $\log W_n(\beta,x)$ (where $x\in\mathbb Z^d$ denotes the starting point of the walk) to a discretized formulation of \eqref{eq:formalKPZ}.

The microscopic (pointwise) behavior of $u_\e(t,x)$ and $h_\e(t,x)$ have been studied in \cite{CCM20,DGRZ18b}. It was shown that if $h_\e^{stat}$ is the stationary solution of the SPDE in \eqref{eq:KPZe}, then in the full $L^2$-region, for all continuous and bounded $h_0$, $t>0,x\in\mathbb R^d$ and as $\e \to 0$, 
\[h_\e(t,x)-h^{stat}_\e(t,x) - \log \bar u(t,x) \cvIP 0,\]
where $\partial_t \bar{u}(t,x) = \frac{1}{2} \Delta \bar{u}(t,x)$ solves the deterministic heat equation with $\bar{u}(0,\cdot) = \exp h_0$. In the case where $h_0\equiv 0$, the fluctuations in the above estimate have been obtained in \cite{CCM19} in a restrained part of the $L^2$-region: jointly for finitely many $t>0,x\in\mathbb R^d$, we have
\begin{equation} \label{eq:CVCCM}
\e^{-\frac{d-2}{2}}\left(h_\e(t,x) - h^{stat}_\e(t,x) \right)\cvlaw \mathscr H(t,x),
\end{equation}
where $\partial_t \mathscr H(t,x) = \frac{1}{2} \Delta \mathscr H(t,x)$ and $\mathscr H(0,\cdot)\eqlaw \mathscr H_{GFF}$ with $\mathscr H_{GFF}$ the Gaussian free field with covariance given by a multiple of the Green function:
\begin{equation}\label{cov:H:GFF}
\mathrm{Cov}\big(\mathscr H_{GFF}(x),\mathscr H_{GFF}(y)\big) = \gamma^2(\beta) \frac{\Gamma(\frac d2-1)}{\pi^{d/2}|x-y|^{d-2}},
\end{equation}
and ${\gamma}^2(\beta)\to \infty$ as $\beta$ approaches the critical $L^2$ parameter (this is the analogue in the continous setting of $\sigma(\beta)$ appearing above).

The fluctuations in \eqref{eq:CVCCM} are in fact related to the results of the present paper. Indeed, by its definition the stationary solution satisfies $h_\e^{stat}(t,x) \eqlaw \log \sZ_\infty(\beta,0)$  for all $t,x$ (see \cite{CCM20,DGRZ18b}), therefore \eqref{eq:CVCCM} shows in particular that in a restricted part of the $L^2$-region,
\begin{equation} \label{eq:CCMpoly}
{\e^{-\frac{d-2}{2}}} \left(\log \sZ_{\e^{-2}}(\beta,0)- \log \sZ_{\infty} (\beta,0)\right)\cvlaw  c\gamma(\beta) G,
\end{equation}
where $G$ is a standard Gaussian and $c>0$ is some constant. Hence Corollary \ref{cor:mainCor} is an extension of this property (in the discrete case) to the full $L^2$-region. We believe that the techniques of the present paper should also apply to the continuous setting in order to extend \eqref{eq:CCMpoly} to the full $L^2$-region. They may moreover help extending \eqref{eq:CVCCM} to the full $L^2$-region by borrowing some further tools from \cite{CNN20}.

Let us now briefly present what is known at the macroscopic level (i.e.\ when integrating against test functions). It has been shown in \cite{CNN20,MSZ16} that in the full weak disorer region of the continuous polymer, the following law of large numbers holds: if $u_0= u_\e(0,\cdot)$ is continuous and bounded, then as $\e\to 0$ and for any test function $f\in \mathcal{C}_c^\infty$,
\begin{equation} \label{eq:LLN}
\int_{\mathbb R ^d} u_\e(t,x) f(x) \dd x \cvIP \int_{\mathbb R ^d} \bar{u}(t,x) f(x) \dd x,
\end{equation}
with $\bar u$ as above. The fluctuations in convergence \eqref{eq:LLN} have been studied in \cite{CNN20,GRZ18} and it has been proved that in the full $L^2$-region of the polymer,
\begin{equation} \label{eq:CV_toEW_intro}
\e^{-\frac{d-2}{2}}\int_{\mathbb R^d} f(x) \left(u_\e(t,x) - \bar{u}(t,x)\right)  \dd x \cvlaw \int_{\mathbb R^d} f(x)\, \kU_1(t,x) \dd x\;,
\end{equation}
with $\mathscr U_1$ solving
the  stochastic heat equation with additive noise (also called the Edwards-Wilkinson (EW) equation):
\begin{equation}\label{eq:EW_GRZ}
\partial_t \kU_1(t,x)= \frac 12 \Delta \kU_1(t,x)+ \bar{u}(t,x) {\gamma}(\beta)  \xi(t,x),\qquad \kU_1(0,x)=0,
\end{equation}
with the same constant $\gamma(\beta)$ as above. 

Moreover, it is also known in the $L^2$-region \cite{CNN20,DGRZ20,MU18} that when $h_0:=h_\e(0,\cdot)$ is continuous and bounded,
\begin{equation} \label{eq:EWKPZintro}
\e^{-\frac{d-2}{2}}\int_{\mathbb R^d} f(x) \left(h_\e(t,x)-\IE[h_\e(t,x)]\right)  \dd x \cvlaw \int_{\mathbb R^d} f(x)\, \kU_{3}(t,x) \dd x,
\end{equation}
with $\mathscr U_3$ solving
the following Edwards-Wilkinson type of  stochastic heat equation with additive noise:
\begin{equation} \label{eq:defU3}
 \partial_t  \kU_3(t,x)=\frac{1}{2} \Delta \kU_3(t,x) + \nabla \log{ \bar{u}(x,\tau)} \cdot\nabla \kU_3(t,x) +\xi(x,\tau)\qquad \kU_3(0,x)=0.
 \end{equation}
The paper \cite{CNN20} focuses on proving \eqref{eq:CV_toEW_intro} and \eqref{eq:EWKPZintro} in the \emph{full} $L^2$-region and the proof builds on a martingale CLT ($\kU_1$ is a Gaussian field) combined with the polymer local limit theorem and homogenization technique introduced in the present paper. Let us also mention that the analogues of \eqref{eq:CV_toEW_intro} and \eqref{eq:EWKPZintro} in the discrete polymer/SHE-KPZ setting has been similtaneously shown to hold in the full $L^2$-region in \cite{LZ20}.

It is further proved in \cite{CNN20} that the rescaled infinite time horizon (log)-partition function converges to the Gaussian Free field:
\begin{equation}
T^{\frac{ (d-2)}{ 4}} \int_{\mathbb{R}^d} f(x)\, \left(\log \sZ_{\infty}(\sqrt T x)-\IE\log \sZ_{\infty}(\sqrt T x)\right) \dd x \cvlaw \int_{\mathbb{R}^d} f(x)\, \mathscr H_{GFF}(x) \dd x,
\end{equation}
and
\begin{equation}
T^{\frac{ (d-2)}{ 4}} \int_{\mathbb{R}^d} f(x)\, \left(\sZ_{\infty}(\sqrt T x)-1\right) \dd x \cvlaw \int_{\mathbb{R}^d} f(x)\, \mathscr H_{GFF}(x) \dd x.
\end{equation}
Again, the variance of {$\langle f,\mathscr H_{GFF}\rangle$} blows up at the $L^2$ critical point and what happens above it is still open.\\

\subsection{The case of dimension $d=2$}
The regularized and discrete SHE and KPZ equation in dimension $d=2$ have been studied a series of papers  \cite{CSZ17b,CD18,G18,CSZ18a,CSZ18b} and there are many similarities with dimension $d\geq 3$. In particular, the analogue Edwards-Wilkinson regime as described in the previous section has been proved in the corresponding full $L^2$-region \cite{CSZ18b,CD18,G18}.

However, the pointwise behavior in the $L^2$-region is slightly different from the $d\geq 3$ case. Let us state what happens in the discrete (polymer) framework. In the $d=2$ case, the $\beta$-regime of interest is obtained by rescaling the temperature as
$\beta = \hat \beta / {\sqrt {\log n}}$. The authors in \cite{CSZ17b} showed that there is an (explicit) critical parameter $ \hat \beta_c$ such that:
\begin{equation} \label{eq:ptwdequals2}
W_n\left(\frac{\hat \beta}{\sqrt {\log n}}\right) \cvlaw \begin{cases}
\exp\left\{c_{\hat \beta} G - \frac {c_{\hat \beta}^2}{2} \right\}& \text{if } \hat \beta < \hat \beta_c,\\
0 & \text{if } \hat \beta \geq \hat \beta_c,
\end{cases}
\end{equation}
where $c_{\hat \beta}^2 = -\log (1-{\hat \beta}^2/{\hat \beta_c}^2)$. Moreover, the phase $\hat\beta \in(0,\hat \beta_c)$ corresponds to the $L^2$-region, i.e.\ the $\hat \beta$-region where $W_n(\hat \beta / \sqrt {\log n})$ remains bounded in $L^2$. Due to the nature of the limit in \eqref{eq:ptwdequals2}, it is not clear how to study fluctuations for the tail of the partition function in the $L^2$-region (while when $d\geq 3$, $W_n \to W_\infty$ almost surely under weak disorder).

What happens in a zoomed window around $\hat \beta_c$ has been investigated in \cite{GQT19,CSZ18a,CSZ19}. It is shown in \cite{CSZ18a} that for $\beta_n^2 = \hat \beta_c^2/\log n - c/(\log n)^{3/2}(1+o(1))$, although the pointwise limit $W_n(\beta_n)$ is null by \eqref{eq:ptwdequals2}, the diffusively rescaled partition function admits a non trivial limit (along a subsequence) in the sense of random measure, i.e. the following measure is tight in $n$ (we write $W_n(\beta,x)$ if the walk is started at $x\in\mathbb Z^2$)
\[\mathbf Z_n(\dd x) = \frac{1}{n} \sum_{y\in \frac{1}{\sqrt n} \mathbb Z^2} W_n\left(\beta_n, y \sqrt n\right) \delta_y(\dd x).
\]
with non-trivial limiting moments of the spatial average $\langle \mathbf Z_n,\phi\rangle$ for compactly supported and continuous $\phi$, see \cite{GQT19,CSZ18a,CSZ19}. 

In dimension $d\geq 3$, the same precise phenomena will not occur at a zoomed window around $\beta_{L^2}$ since the spatial average of $W_n(\beta,x\sqrt n)$ converges to a deterministic constant in the whole weak disorder region, cf. \eqref{eq:LLN} in the continuous setting. At $\beta=\beta_c$ however, $W_n$ converges pointwise to zero and similar features might occur for spatial averages, but for now this question seems to us out of reach with the existing tools (in particular the value of $\beta_c$ is still unknown for $d\geq 3$). Nevertheless, it is interesting to wonder whether the fluctuations in \eqref{eq:LLN} would remain of Gaussian type at $\beta_{2}$ (with a different scaling factor than in \eqref{eq:CV_toEW_intro}), or if they would have similar features to spatial averages of the partition function when $d=2$ at $\hat \beta_c$.\\

\section{Idea of the proof}
\subsection{A central limit theorem for martingales} \label{subsec:CLT}
As in \cite{CL17}, the main tool to prove Theorem \ref{th:mainTheorem} is the following theorem:
\begin{theorem}[Corollary 3.2. in \cite{CL17}] \label{th:CLTforMartingale}
Let $(M_n)_{n\geq 0}$ be a martingale defined on a probability space $(\Omega,\mathcal{F},\IP)$, with adapted filtration $(\mathcal F _n)_{n\geq 0}$, $M_0=0$, which is bounded in $L^2$. Let $D_{k+1} = M_{k+1}-M_{k}$ for all $k\geq 0$ and let $M_\infty=\lim_{n\to\infty} {M_n}$ be the a.s. limit of $M_n$. Also define:
\begin{equation} \label{eq:asymptVariance}
v_n^2 = \IE\big[(M_\infty - M_n)^2\big] = \IE \sum_{k=n}^\infty D_{k+1}^2\ .
\end{equation}
Suppose that $v_n$ is always positive and that:
\begin{enumerate}[label=(\alph*)]
\item There exists a non-negative and finite random variable $V$, such that
\[
V_n^2 = \frac{1}{v_n^2} \sum_{k=n}^\infty \IE\left[D_{k+1}^2\middle|\mathcal F_k\right] \overset{\IP}{\longrightarrow} V^2;
\]
\item The following conditional Lindeberg condition holds:
\[\forall \epsilon >0,\quad \frac{1}{v_n^2}\sum_{k=n}^\infty \IE\left[D_{k+1}^2 \mathbf{1}_{\{|D_{k+1}|>\epsilon v_n\}} \middle| \mathcal F_k\right] \overset{\IP}{\longrightarrow} 0.\]
\end{enumerate}
Then,
\begin{equation}\label{eq:CLstableCV}
\frac{M_\infty-M_n}{v_n}\cvlaw V\ G,
\end{equation}
where $G$ is a standard Gaussian random variable which is independent of $V$. If, additionally, $V\neq 0$ a.s., then
\begin{equation}\label{eq:CLmixingCV}
\frac{M_\infty-M_n}{V_n}\cvlaw G.
\end{equation}
Moreover, convergence \eqref{eq:CLstableCV} is stable and convergence \eqref{eq:CLmixingCV} is mixing.
\end{theorem}
\noindent
To prove Theorem \ref{th:mainTheorem}, we show that condition (a) and (b) hold for $M_n = W_n$ and some suited $V$. The proof of condition (b) is postponed to Section \ref{subsection:lindebergCondition}. Our main focus will be condition (a): if we let
\[D_{k+1}= W_{k+1}-W_k,\]
then, by estimate \eqref{eq:AsymptVarianceValue}, condition (a) follows from: 
\begin{theorem} \label{th:CV_bracket}
For all $\beta\in(0,\beta_2)$, as $n\to\infty$,
\begin{equation} \label{eq:mainCV}
s_n^2:=n^{(d-2)/2} \sum_{k\geq n} \IE\big[D_{k+1}^2|\mathcal{F}_k\big] \cvLone \sigma^2 W_\infty^2.
\end{equation}
\end{theorem}
The structure of the proof for Theorem \ref{th:CV_bracket} is described in Section \ref{subsection:ideaOfProof}. We now turn to the proof of the main theorem.
\begin{proof}[Proof of Theorem \ref{th:mainTheorem}]
It follows directly from Theorem \ref{th:CV_bracket} that as $n\to\infty$,
\begin{equation} \label{eq:AsymptVarianceValue}
v_n^2 := \IE\big[(W_\infty-W_n)^2\big] \sim \sigma^2\, \IE\big[W_\infty^2\big] \,n^{\frac{d-2}{2}},
\end{equation}
where $v_n$ is as in \eqref{eq:asymptVariance} (we also refer the reader to Proposition 2.1 of \cite{CL17} for a more direct argument). Hence, Theorem \ref{th:CV_bracket} implies condition (a) with limiting variable $V$ given by 
\[V = \IE\big[W_\infty^2\big]^{-1/2} W_\infty.\]
Combined with condition (b), Theorem \ref{th:CLTforMartingale} implies convergence \eqref{eq:CLstableCV} which in turn gives \eqref{eq:stableMainTh}. Then, to get \eqref{eq:mixingMainTh} from \eqref{eq:CLmixingCV}, observe that $V_n/V$ and $W_\infty/W_n$ both converge in probability $1$, so that by simple multiplication, one can replace $V_n$ by $V$ in \eqref{eq:CLmixingCV}, and then $W_\infty$ by $W_n$ to obtain \eqref{eq:mixingMainTh}.
\end{proof}

\subsection{Structure of the proof of Theorem \ref{th:CV_bracket}} \label{subsection:ideaOfProof}
By a standard computation, the summand in \eqref{eq:mainCV} satisfies (recall the definiton of $e_k$ in \eqref{eq:defWnek})
\begin{equation} \label{eq:starting_point}
\IE\big[D_{k+1}^2|\mathcal{F}_k\big] = \kappa_2(\beta) \sum_{x\in\mathbb{Z}^d} \DE\left[e_k \mathbf{1}_{\{S_{k+1}=x\}} \right]^2,
\end{equation}
where
\[\kappa_2(\beta) = e^{\lambda_2(\beta)}-1.\]
In order to study the right-hand side of \eqref{eq:starting_point}, we appeal to the following theorem:

\begin{theorem}[Local limit theorem for polymers in the $L^2$-region \cite{S95,V06}] \label{th:L2LLT}
Let $\beta\in(0,\beta_2)$ and $\alpha>0$. For any sequence $(l_k)_{k\geq 0}$, verifying that $l_k \to \infty$ and $l_k = o(k^a)$ for some $a<1/2$,
\begin{equation}
\DE\left[e_k|~S_{k+1}=x\right] = W_{l_k} \, \overleftarrow{W}_{k+1,l_k}^x + \delta_k^{x},
\end{equation}
where $\overleftarrow{W}^y_{k,l} = \DP_y\left[\exp\left(\sum_{i=1}^l\omega(k-i,S_i)-l\lambda(\beta)\right)\right]$ is the time-reversed  normalized partition function, and where, as $k \to \infty$,
\begin{equation} \label{eq:L2ErrorVanish}
\sup_{|x|\leq \alpha \sqrt k} \IE\left[ \left|\delta_k^{x} \right|^2 \right] \to 0.
\end{equation}
\end{theorem}
\begin{remark}
Note that we have reformulated the result with endpoint distribution at time $k+1$, for a polymer measure of horizon $k$, so that the time-reversed partition function $\overleftarrow{W}_{k+1,l_k}^x$ does not take into account the environment at time $k+1$.
\end{remark}
By the local limit theorem for polymers,
\begin{align}
s_n^2 & = \kappa_2(\beta)  n^{(d-2)/2} \sum_{k\geq n}  \sum_{x\in\mathbb{Z}^d} \DE\left[e_k \mathbf{1}_{\{S_{k+1}=x\}} \right]^2 \nonumber \\
& =: A_n + B_n + C_n + F_n,
\end{align}
where:
\begin{equation}
A_n = \kappa_2(\beta) \, n^{(d-2)/2} \sum_{k\geq n}   \sum_{|x|\leq \alpha \sqrt{k}} \left(W_{l_k} \, \overleftarrow{W}_{k+1,l_k}^x\right)^2 \DP(S_{k+1} = x)^2, 
\end{equation}
and:
\begin{align*}
& B_n = 2 \kappa_2(\beta)\, n^{(d-2)/2} \sum_{k\geq n}   \sum_{|x|\leq \alpha \sqrt{k}} \delta_k^x \  W_{l_k} \, \overleftarrow{W}_{k+1,l_k}^x\, \DP(S_{k+1} = x)^2, \\
& C_n = \kappa_2(\beta)\, n^{(d-2)/2} \sum_{k\geq n}   \sum_{|x|\leq \alpha \sqrt{k}} \left(\delta_k^x\right)^2 \DP(S_{k+1}=x)^2,\\
& F_n= \kappa_2(\beta)n^{(d-2)/2} \sum_{k\geq n} \sum_{|x| > \alpha\sqrt{k}} \DE \left[e_k \mathbf{1}_{\{S_{k+1} = x\}}\right]^2.
\end{align*}
Section \ref{subsection:removingNegligeable} is dedicated to showing that $B_n$, $C_n$ and $F_n$ all vanish in $L^1$ norm. Turning to $A_n$, we note that $\overleftarrow{W}_{k+1,l_k}^x$ and $\overleftarrow{W}_{k+1,l_k}^y$ are independent whenever $|x-y|_1 > l_k$, so that, by some homogenization argument, we can show that
\begin{align}
\label{eq:approx_An} A_n &\approx \kappa_2(\beta)\, n^{(d-2)/2} \sum_{k\geq n}   \sum_{|x|\leq \alpha \sqrt{k}} W_{l_k}^2\, \IE \left [ \left(\overleftarrow{W}_{k+1,l_k}^x\right)^2\right] \DP(S_{k+1} = x)^2\\ \nonumber
& \to \sigma^2 W_{\infty}^2,
\end{align}
as $n\to\infty$ and $\alpha\to\infty$ in this order. Approximation \eqref{eq:approx_An} is justified in Section \ref{subsection:homogenResult}, while, letting $\overline{A}_n$ denote the RHS of \eqref{eq:approx_An}, convergence in the second line is proved in Proposition \ref{expect-converge}.

\section{Proof}
\paragraph{Notations}
\begin{itemize}
\item $|\cdot|$ stands for the Euclidean norm on $\mathbb{R}$ or $\mathbb{R}^d$,
\item $|\cdot|_1$ stands for the usual $L^1$-norm on $\mathbb{R}^d$,
  \item $B(r)$ denotes the closed ball of radius $r$ in the Euclidean norm,
  \item $\DE^{\otimes 2}$ and $\DP^{\otimes 2}$ stand for resp.\  the expectation and the probability measure for two independent simple random walks $S$ and $\widetilde{S}$,
  \item $\IE_k[\cdot]=\IE[\cdot|~\mathcal{F}_k]$,
    \item { $e_n = e^{\beta\sum_{i=1}^n\omega(i,S_i)-n\lambda(\beta)}$.}
\item $N_{m,k}=\sum_{i=m}^k\mathbf{1}_{\{S_i=\tilde{S}_i\}}$ denotes the time overlap of two different paths $S$ and $\widetilde{S}$ from time $m$ to $k$. When $m=1$, we simply write $N_k$. 
\end{itemize}
\subsection{Some tools}
\begin{theorem}[Local central limit theorem for the simple random walk \cite{L13}] \label{th:StandardLLT}

For all $x$ such that $\DP(S_k=x)\neq 0$, as $k\to\infty$,
\begin{gather}
\DP(S_{k} = x) = 2\left(\frac{d}{2\pi k}\right)^{d/2}e^{-d\frac{x^2}{2k}} + O\left(k^{-{(d+2)}/2}\right),
\end{gather} 
where the big O term is uniform in $x$.
In particular,
\begin{equation}
\DP(S_{2k} = 0) \sim 2\left(\frac{d}{4\pi k}\right)^{d/2}.\label{eq:LCLTat0}
\end{equation}
\end{theorem}

\begin{proposition} There exists $\mathcal{Z}_d>0$, such that as $n\to\infty$,
\begin{equation} \label{eq:useful_sum}
  n^{(d-2)/2} \sum_{k\geq n}   \sum_{x\in\Z^d} \DP(S_{k+1} = x)^2 \to \mathcal{Z}_d < \infty.
\end{equation}
\begin{proof}
{Since $P_0(S_k=x)=P_x(S_k=0)$, we have $\sum_{x\in\Z^d} \DP(S_{k+1} = x)^2 = P(S_{2(k+1)} = 0)$. The statement of the proposition then follows from \eqref{eq:LCLTat0}.}
\end{proof}
\end{proposition} We will also require the following technical proposition:
\begin{proposition}\label{prop:Bridge}  Let $v$ be a non-negative bounded function on $\mathbb{Z}^d$ with $d\geq 3$, such that:
\[\sup_{x\in\mathbb{Z}^d} \DE_x \left[ {e}^{\sum_{k=1}^{\infty}v(S_{2k})} \right] < \infty.\]
Then, there exists $C\in(0,\infty)$, such that for all $n\geq 0$,
\[ \DE_0 \left[ e^{\sum_{k=1}^{n} v(S_{2k})} \middle| S_{2(n+1)} = 0 \right] \leq C.\]
\end{proposition}
\begin{proof}
{Choose $f(y)=\mathbf{1}_{\{ y=0 \}}$ and $x=0$ in Lemma \ref{lem:Vargas} below and conclude using estimate \eqref{eq:LCLTat0}.}
\end{proof}
\noindent The following is an analogue to Lemma 3.3 in \cite{V06}:
\begin{lemma} \label{lem:Vargas}
Under the assumptions of Proposition \ref{prop:Bridge}, there exists a constant $C\in(0,\infty)$, such that for all non-negative function $f$ on $\mathbb{Z}^d$, for all $n\geq 0$,
\[\sup_{\substack{x\in\mathbb{Z}^d \\ |x|_1: \text{even}}}\DE_x\left[e^{\sum_{i=1}^n v(S_{2i})}f(S_{2n})\right] \leq \frac{C}{n^{d/2}} \sum_{\substack{y\in\mathbb{Z}^d \\ |y|_1: \text{even}}} f(y).\]
\end{lemma}
\begin{proof}
We repeat the argument of \cite{V06}, which is a little simpler in our case. Let $A=\{x\in\mathbb{Z}^d,\, |x|_1 \text{ is even}\}$ be the underlying graph of $(S_{2i})$. We also let $x\sim y$ denote the fact that $x$ and $y$ are nearest neighbors in $A$, and $p^{(2)}(x,y)$ be the transition kernel $\DP(S_2=y|S_0=x)$.
For all $x\in A$, define $h(x)=\DE_x\left[e^{\sum_{i=1}^\infty v(S_{2i})}\right]$; then $h$ satisfies
\[h(x)=\sum_{y\sim x} p^{(2)}(x,y) e^{v(y)} h(y).\]
Hence, similarly to Doob's h-transform, the kernel
\[K(x,y) = \frac{h(y)}{h(x)} e^{v(y)} p^{(2)}(x,y)\]
defines a probability transition kernel of a Markov chain on $A$, for which $m(x) = h(x)^2 e^{v(x)}$ is a reversible measure. 

By assumption, $m$ is uniformly bounded and bounded away from $0$, and there exists a constant $c\in(0,\infty)$, such that $K(x,y)\geq cp^{(2)}(x,y)$ for all $x,y\in A$. As by Theorem 4.18 in \cite{W00}, $S_{2n}$ satisfies the $d$-isoperimetric inequality (cf. pp. 39-40 therein) on A, this implies that $K$ also satisfies it.

Therefore, we get from Corollary 14.5 in \cite{W00} that there exists a finite $C$, such that
\[
\frac{1}{h(x)} \DE_x\left[h(S_{2n}) e^{\sum_{i=1}^n v(S_{2i})} f(S_{2n})\right] =  \sum_{y\in A} K^{(n)}(x,y) f(y)\leq \frac{C}{n^{d/2}} \sum_{y\in A} f(y).
\]
This in turn implies the lemma by our assumptions.
\end{proof}

\subsection{Removing the negligeable terms} \label{subsection:removingNegligeable}
The following proposition justifies that $F_n$ from Section \ref{subsection:ideaOfProof} is negligible in $L^1$-norm.
\begin{proposition} \label{prop_neglectFar} We have :
\begin{equation*}
\lim_{\alpha \to \infty} \limsup_{n\to\infty}\, \IE\left[n^{(d-2)/2} \sum_{k\geq n} \sum_{|x|\geq \alpha\sqrt{k}} \DE \left[e_k \mathbf{1}_{\{S_{k+1} = x\}}\right]^2 \right] = 0.
\end{equation*}
\end{proposition}
\begin{proof}
The finite positive constants $C$ that will arise in the paper may change from line to line, but they will not depend on any varying parameter. We write:
\begin{align*}
&\IE n^{(d-2)/2}  \sum_{k\geq n} \sum_{|x|\geq \alpha\sqrt{k}} \DE \left[e_k \mathbf{1}_{\{S_{k+1} = x\}}\right]^2\\
& =  n^{(d-2)/2} \sum_{k\geq n} \DE^{\otimes 2} \left[e^{\lambda_2 N_k} \mathbf{1}_{S_{k+1}=\widetilde{S}_{k+1}} \mathbf{1}_{|S_{k+1}|\geq \alpha \sqrt{k}}   \right]\\
&=  n^{(d-2)/2} \sum_{k\geq n} \DE^{\otimes 2}\left[e^{\lambda_2 N_k}\mathbf{1}_{|S_{k+1}|\geq\alpha \sqrt{k}}\,\middle|\,S_{k+1}=\widetilde{S}_{k+1} \right] \DP(S_{k+1}=\widetilde{S}_{k+1}).
\end{align*}
As $\DE^{\otimes 2}[e^{\lambda_2 N_\infty}]=\IE[W_\infty^2]$ is given by the RHS of \eqref{eq:SndMomentWinfty}, we can apply H\"older's inequality, for $p^{-1} + q^{-1} = 1$ with the only constraint that $p\lambda_2(\beta) < \log{(1/\pi_d)}$, and use Lemma \ref{lem:Vargas} (note that $S_k-\widetilde{S}_k \eqlaw S_{2k}$), to get that
\[
\DE^{\otimes 2}\left[e^{\lambda_2 N_k}\mathbf{1}_{|S_{k+1}|\geq \alpha \sqrt{k}}\middle|S_{k+1}=\widetilde{S}_{k+1} \right]\leq C \DE^{\otimes 2}\left[\mathbf{1}_{|S_{k+1}|\geq \alpha \sqrt{k}}\middle|S_{k+1}=\widetilde{S}_{k+1} \right]^{1/q}
\]
Then, by the local central limit theorem (Theorem \ref{th:StandardLLT}), there exists some positive constants $C$, such that for large enough $n$,
\begin{align*}
\DE^{\otimes 2}\left[\mathbf{1}_{\{|S_{k+1}|\geq \alpha \sqrt{k}\}}\middle|S_{k+1}=\widetilde{S}_{k+1} \right] & \leq \sum_{|x|\geq \alpha \sqrt{k}} \frac{\DP(S_{k+1}=x)^2}{ \DP^{\otimes 2} (S_{k+1}=\widetilde{S}_{k+1})}\\
&\leq \frac{ \max_{|x|\geq \alpha\sqrt{k}}\DP(|S_{k+1}|=x)}{ \DP^{\otimes 2} (S_{k+1}=\widetilde{S}_{k+1})}\\
&\leq  C\alpha^{-2}.
\end{align*}
We thus obtain that, for large enough $n$,
\begin{align*}
&\IE n^{(d-2)/2} \sum_{k\geq n} \sum_{|x|\geq \alpha\sqrt{k}} \DE \left[e_k \mathbf{1}_{\{S_{k+1} = x\}}\right]^2\\
  & \leq C\alpha^{-2/q} \, n^{(d-2)/2} \sum_{k\geq n} \DP(S_{k+1}=\widetilde{S}_{k+1})\\
  & = C\alpha^{-2/q} \, n^{(d-2)/2} \sum_{k\geq n} \sum_{x\in\mathbb{Z}^d}\DP(S_{k+1}=x)^2\\
&\leq C\alpha^{-2/q},
\end{align*}
where the last inequality follows from \eqref{eq:useful_sum} and where the last term vanishes as $\alpha\to\infty$, which concludes the proof.
\end{proof}

\begin{proposition}\label{neglect B and C}
  As $n\to\infty$,
\begin{gather*}
B_n =  2 \kappa_2(\beta)\, n^{(d-2)/2} \sum_{k\geq n}   \sum_{|x|\leq \alpha \sqrt{k}} \delta_k^x \, W_{l_k} \  \overleftarrow{W}_{k+1,l_k}^x \DP(S_{k+1} = x)^2  \cvLone 0.\\
C_n = \kappa_2(\beta)n^{(d-2)/2} \sum_{k\geq n}   \sum_{|x|\leq \alpha \sqrt{k}} \left(\delta_k^x\right)^2 \DP(S_{k+1}=x)^2 \cvLone 0.
\end{gather*}
\end{proposition}
\begin{proof}
From independence of $ W_{l_k}$ and $\overleftarrow{W}_{k+1,l_k}^x$ and since the law of $\overleftarrow{W}_{k+1,l_k}^x$ does not depend on $x$, we obtain from Cauchy-Schwarz inequality: 
\[\IE\left[\left|\delta_{l_k}^x W_{l_k} \  \overleftarrow{W}_{k+1,l_k}^x\right| \right] \leq \IE\left[W_{l_k}^2\right] \IE\left[\left(\delta_k^x\right)^2\right]^{1/2}.\]
The first term of the right-hand side is bounded in the $L^2$-region, so the convergence for $B_n$ simply follows from \eqref{eq:L2ErrorVanish} and  \eqref{eq:useful_sum}. $C_n$ is treated in the same way.
\end{proof}

\subsection{The homogenization result} \label{subsection:homogenResult}
This section is dedicated to proving the next proposition, which justifies approximation \eqref{eq:approx_An}:
\begin{proposition} \label{prop_homogen}
Let $\overline{A}_n$ denote the right-hand side of equation \eqref{eq:approx_An}. Then, for any $\alpha>0$,
\begin{equation*}
\lim_{n\to\infty}\IE |A_n-\overline{A}_n|=0.
\end{equation*}
\end{proposition}

\noindent To show the result of the proposition, it is enough to prove that, as $k\to\infty$,
\begin{equation} \label{eq:CV_Mk}
M_k := k^{d/2}\sum_{|x|\leq \alpha \sqrt{k}} ({Y}_{k,x}-\IE {Y}_{k,x})\DP (S_{k+1}=x)^2 \cvLone 0,
\end{equation}
where $Y_{k,x}=(\overleftarrow{W}_{k+1,l_k}^x)^2$. Indeed, for large $k$, we have $l_k<k/2$ and so $W_{l_k}$ and $\overleftarrow{W}_{k+1,l_k}^x$ are independent. Hence:
\begin{align*}
& \IE \left| A_n - \overline{A}_n\right|\\
&  \leq \kappa_2\, n^{(d-2)/2} \sum_{k\geq n}  \IE\left[ W_{l_k}^2\right] \IE \left[ \left| \sum_{|x|\leq \alpha \sqrt{k}} ({Y}_{k,x}-\IE {Y}_{k,x})  \DP(S_{k+1} = x)^2 \right| \right]  \\
& \leq \kappa_2\, \IE[W_\infty^2]\,  n^{(d-2)/2}\sum_{k\geq n} k^{-d/2}\, \IE[|M_k|],
\end{align*}
where the last term vanishes, as $n\to\infty$, if one  assumes convergence \eqref{eq:CV_Mk}. To prove this  convergence, we rely on a truncation technique:
\begin{lemma} \label{lem_cvL2Trunc}
Let $\tilde{Y}_{k,x}=Y_{k,x}\land (k^{d/2}\, l_k^{-d})$. We have,
\[\lim_{k\to\infty}\IE \left(k^{d/2}\sum_{|x|\leq \alpha \sqrt{k}} \left(\tilde{Y}_{k,x}-\IE \tilde{Y}_{k,x}\right)\DP (S_{k+1}=x)^2\right)^2=0.\]
\end{lemma}
\begin{proof}
  Since $Y_{k,x}$ and $Y_{k,y}$ are independent for $|x-y|_1\geq l_k$,
  \begin{align*}
    & \IE \left(k^{d/2}\sum_{|x|\leq \alpha \sqrt{k}} (\tilde{Y}_{k,x}-\IE \tilde{Y}_{k,x})\DP (S_{k+1}=x)^2\right)^2\\
    & =k^{d}\sum_{|x|,|y|\leq \alpha \sqrt{k}}  \IE \left(\tilde{Y}_{k,x}-\IE \tilde{Y}_{k,x}\right)\left(\tilde{Y}_{k,y}-\IE \tilde{Y}_{k,y}\right)\DP (S_{k+1}=x)^2\DP (S_{k+1}=y)^2\\
    &\leq C k^{-d}\sum_{|x-y|_1\leq l_k,|x| \leq \alpha \sqrt{k}}  \IE \left(\tilde{Y}_{k,x}-\IE \tilde{Y}_{k,x}\right)\left(\tilde{Y}_{k,y}-\IE \tilde{Y}_{k,y}\right)\\
    &\leq C k^{-d}\sum_{|x-y|_1\leq l_k,|x|\leq \alpha \sqrt{k}}  \IE \left(\tilde{Y}_{k,0}-\IE \tilde{Y}_{k,0}\right)^2,\\
  \end{align*}
  where we have used the local central limit theorem in the first inequality; we used the Cauchy-Schwarz inequality and the fact that $\tilde{Y}_{k,x}$ are identically distributed with respect to $x$ in the last one. This is further bounded from above by
  \begin{align*}
    C k^{-d/2}\,l_k^d\,  \IE \left(\tilde{Y}_{k,0}-\IE \tilde{Y}_{k,0}\right)^2& \leq C k^{-d/2}\,l_k^d\,  \IE \tilde{Y}_{k,0}^2\\
    & \to 0, 
  \end{align*}
  as $k\to\infty$, where, observing that the family $(Y_{k,0})_k$ is uniformly integrable since $W_k^2$ converges in $L^1$, the convergence in the second line is justified by the following lemma.
\end{proof}

\begin{lemma}
Let $(X_k)_{k\in\mathbb{N}}$ be a non-negative, uniformly integrable family of random variables. Then, for any sequence $a_k\to\infty$, $a_k^{-1} \IE\left[\left(X_k \wedge a_k\right)^2\right] \to 0$.
\end{lemma}
\begin{proof} By property: $x\IP \left(X_k\geq x \right) \leq \IE[X_k \mathbf{1}_{\{X_k\geq x\}}]$, we have
\begin{align*}
a_k^{-1} \IE\left[\left(X_k \wedge a_k\right)^2\right]&= a_k^{-1} \int^{a_k}_0 2x\IP \left(X_k\geq x \right)dx\\
&\leq 2 a_k^{-1} \int^{a_k}_0 \sup_{k\in\mathbb N} \IE[X_k \mathbf{1}_{\{X_k\geq x\}}]dx \to 0,
\end{align*}
as $k\to\infty$, since $\sup_k \IE[X_k \mathbf{1}_{\{X_k\geq x\}}]\to 0$ as $x\to\infty$, by uniform integrability.
\end{proof}

The next lemma will be used in order to remove the truncation:
\begin{lemma} \label{lem:remove_trunc} We have:
\begin{equation}\label{remove_trunc1}
    k^{d/2}\sum_{|x|\leq \alpha \sqrt{k}} \left(Y_{k,x}- \tilde{Y}_{k,x}\right)\DP (S_{k+1}=x)^2 \cvLone 0.
\end{equation}
Moreover,
\begin{equation}\label{remove_trunc2}
 \lim_{k\to\infty} k^{d/2}\sum_{|x|\leq \alpha \sqrt{k}} \left(\IE Y_{k,x}- \IE \tilde{Y}_{k,x}\right)\DP (S_{k+1}=x)^2 = 0.
\end{equation}
  \end{lemma}
\begin{proof}
  Note that
  \[  \IE[|Y_{k,0}-\tilde{Y}_{k,0}|] \leq \IE\left[Y_{k,0};~Y_{k,0}>k^{d/2}l_k^{-d} \right]\to 0.\]
  Thus, combining with the local CLT (Theorem \ref{th:StandardLLT}), we obtain \eqref{remove_trunc1} and \eqref{remove_trunc2}.

  \end{proof}
 
Finally, putting together Lemma \ref{lem_cvL2Trunc} and Lemma \ref{lem:remove_trunc}, we get that  $M_k \cvLone 0$, as desired.



\subsection{Proof of Theorem \ref{th:CV_bracket}}
Combined to propositions of the last two sections, the following theorem entails Theorem \ref{th:CV_bracket}:
\begin{proposition}\label{expect-converge} Recall that $\overline{A}_n$ denotes the RHS of equation \eqref{eq:approx_An}. With $\sigma^2$ defined as in \eqref{eq:defSigma},
\[\lim_{\alpha\to\infty}\limsup_{n\to\infty}\IE\left|\overline{A}_n -\sigma^2 W^2_\infty\right|=0.\]
\end{proposition}
\noindent 
\begin{proof}
Note first that by \eqref{eq:useful_sum},
\begin{align*}
  &\quad \IE \left|\overline{A}_n-\kappa_2(\beta)\, n^{(d-2)/2} \sum_{k\geq n}   \sum_{|x|\leq \alpha \sqrt{k}} W_{\infty}^2\, \IE \left [ \left(\overleftarrow{W}_{k+1,l_k}^x\right)^2\right] \DP(S_{k+1} = x)^2 \right|\\
  &\leq C \sup_{k\geq n} \IE|W_{\infty}^2-W_{l_k}^2|\to 0, 
\end{align*}
as $n\to\infty$, and 
\begin{align*}
  & \IE \left|\kappa_2\, n^{(d-2)/2} \sum_{k\geq n}   \sum_{|x|\leq \alpha \sqrt{k}} W_{\infty}^2\, \left(\IE \left [ \left(\overleftarrow{W}_{k+1,l_k}^x\right)^2\right]-\IE W_{\infty}^2\right) \DP(S_{k+1} = x)^2 \right|\\
  &\leq C \sup_{k\geq n} \IE|W_{\infty}^2-W_{l_k}^2| \to 0.
\end{align*}
Moreover, by Proposition~\ref{prop_neglectFar}, we have
\[
  \lim_{\alpha\to\infty}\limsup_{n\to\infty}\IE \left|\kappa_2(\beta)\, n^{(d-2)/2} \sum_{k\geq n}   \sum_{|x|> \alpha \sqrt{k}} W_{\infty}^2\, \IE W_{\infty}^2 \DP(S_{k+1} = x)^2 \right|=0.
\]
Therefore, it suffices to show that as $n\to\infty$,
\[
\kappa_2(\beta)\, n^{(d-2)/2} \sum_{k\geq n}   \sum_{x\in\Z^d}  \IE W_{\infty}^2 \DP(S_{k+1} = x)^2\to\sigma^2,
\]
where 
\begin{equation} \label{eq:defSigma}
{\sigma^2=}\sigma^2(\beta) = \frac{(1-\pi_d)(e^{\lambda_2(\beta)}-1)}{1-\pi_d e^{\lambda_2(\beta)}}\mathcal{Z}_d.
\end{equation}
Recalling $\kappa_2(\beta)=e^{\lambda_2(\beta)}-1$ and $\IE W^2_{\infty}=\frac{1-\pi_d}{1-\pi_de^{\lambda_2(\beta)}}$ (cf.\  \eqref{eq:SndMomentWinfty}), this follows from convergence \eqref{eq:useful_sum}.
\end{proof}
\subsection{Proof of condition (b): the Lindeberg condition} \label{subsection:lindebergCondition}
Given the asymptotics of $v_n$ in \eqref{eq:AsymptVarianceValue}, condition (b) of Theorem \ref{th:CLTforMartingale} follows from the following proposition:
\begin{proposition}[Lindeberg condition]
  For any $\eps>0$,
  $$n^{(d-2)/2}\sum_{k\geq n} \IE_k\left[ D^2_{k+1}{\bf 1}_{\{n^{\frac{d-2}{4}}|D_{k+1}|>\eps \}}\right]\cvLone 0.$$
\end{proposition}
\begin{proof}
We first observe that it is enough to prove that
   \begin{align}\label{fixed k Lind}
    \lim_{k\to\infty} k^{d/2} \IE\left[ D^2_{k+1}{\bf 1}_{\{k^{\frac{d-2}{4}}|D_{k+1}|>\eps \}}\right]=0.
   \end{align}
Indeed, since $n^{\frac{d-2}{4}}|D_{k+1}|>\eps$ implies $k^{\frac{d-2}{4}}|D_{k+1}|>\eps$ for $k\geq n$, we would have,  assuming \eqref{fixed k Lind},
  \begin{align*}
    &\quad \limsup_{n\to\infty}\, \IE\left[n^{\frac{d-2}{2}} \sum_{k\geq n }\IE_k\left( D^2_{k+1}{\bf 1}_{\{n^{\frac{d-2}{4}}|D_{k+1}|>\eps \}}\right)\right]\\
    & \leq   \limsup_{n\to\infty} \,n^{\frac{d-2}{2}} \sum_{k\geq n }k^{-d/2}\, \IE\left[ k^{d/2} D^2_{k+1}{\bf 1}_{\{k^{\frac{d-2}{4}}|D_{k+1}|>\eps \}}\right]\\
    &\leq C \limsup_{k\to\infty}\, k^{d/2}\, \IE\left[ D^2_{k+1}{\bf 1}_{\{k^{\frac{d-2}{4}}|D_{k+1}|>\eps \}}\right]=0,
  \end{align*}
where the third inequality comes from the boundedness of $n^{\frac{d-2}{2}} \sum_{k\geq n }k^{-d/2}$. We now focus on showing \eqref{fixed k Lind}.\\

Using $S_k-\widetilde{S}_k \eqlaw S_{2k}$, we may write
\begin{align*}
\IE D_{k+1}^2&= \kappa_2 \DE^{\otimes 2}\left[
e^{\lambda_2 N_k} \mathbf{1}_{\{S_{k+1}=\tilde{S}_{k+1}\}}\right]\\
&=\kappa_2\, \DE^{\otimes 2}\left[
e^{\lambda_2 N_k}|~S_{k+1}=\tilde{S}_{k+1}\right]\DP^{\otimes 2}(S_{k+1}=\tilde{S}_{k+1})\\
&=\kappa_2\, \DE\left[
e^{\lambda_2\sum_{i=1}^{k}\mathbf{1}_{\{S_{2i}=0\}}}|~S_{2(k+1)}=0\right]\DP(S_{2(k+1)}=0)
\end{align*}
By Proposition~\ref{prop:Bridge} and the local CLT, we have $k^{d/2}\IE D_{k+1}^2=O(1)$. Thus, applying Markov's inequality, we get 
  \begin{align}
    \IP\left(k^{\frac{d-2}{4}}|D_{k+1}|>\eps\right)&=\IP\left(k^{\frac{d}{2}}D_{k+1}^2>\eps^2 k\right)\nonumber\\
    &\leq \frac{1}{\eps^2 k} k^{\frac{d}{2}}\IE D_{k+1}^2\to 0,\label{prob-0}
  \end{align}
as $k\to\infty$. 

In order to prove \eqref{fixed k Lind}, we will rely on estimate \eqref{prob-0} and uniform integrability properties. We will need the following simple lemma.
  \begin{lemma}\label{prod of UI}
    Let $\{X_n\},\{Y_n\}$ be independent uniformly integrable families of random variables. Then $\{X_n Y_n\}$ is also uniformly integrable.
  \end{lemma}
  \begin{proof}
    Let us note that $|X_n Y_n|\geq t$, then $|X_n|\geq \sqrt{t}$ or $|Y_n|\geq \sqrt{t}$. Thus,
    \begin{align*}
      &\IE[ |X_n Y_n|;~|X_n Y_n|\geq t]\\
      &\leq \IE[ |X_n Y_n|;~|X_n|\geq \sqrt{t}]+\IE[ |X_n Y_n|;~|Y_n|\geq \sqrt{t}]\\
      &= \IE[|Y_n|]\IE[ |X_n| ;~|X_n|\geq \sqrt{t}]+\IE[|X_n|]\IE[ |Y_n|;~|Y_n|\geq \sqrt{t}],
      \end{align*}
  which uniformly goes to $0$ as $t\to\infty$.
  \end{proof}
  For all $k\in\N$ and $x\in\Z^d$, we write $\eta_k(x)=e^{\beta \omega(k+1,x)-\lambda(\beta)}-1$. Note that
  $$\IE \eta_k(x)=0 \text{ and }\IE \eta_k(x)^2=\kappa_2(\beta).$$
We decompose,
\begin{align*}
      D_{k+1}&= W_{k+1}-W_{k}\\
      &=\sum_{|x|\leq \alpha \sqrt{k}} \DE[ e_k \mathbf{1}_{\{S_{k+1} = x\}}]\eta_k(x)\\
      &\hspace{6mm}+\sum_{|x|> \alpha\sqrt{k}} \DE[ e_k \mathbf{1}_{\{S_{k+1} = x\}}]\eta_k(x),
\end{align*}
and first observe that by proposition~\ref{prop_neglectFar}, we have
    $$\lim_{\alpha\to\infty}\limsup_{k\to\infty} k^{d/2} \IE \left[\left(\sum_{|x|> \alpha\sqrt{k}} \DE[ e_k \mathbf{1}_{\{S_{k+1} = x\}}]\eta_k(x)\right)^2\right]=0.$$
    
Then, if we let $(l_k)_k$ be any positive sequence satisfying the conditions of Proposition~\ref{th:L2LLT}, we can write
    \begin{align*}
      &\quad \sum_{|x|\leq \alpha \sqrt{k}} \DE[ e_k \mathbf{1}_{\{S_{k+1} = x\}}]\eta_k(x)\\
      &=\sum_{|x|\leq \alpha \sqrt{k}} W_{l_k} \overleftarrow{W}^x_{k+1,l_{k}} \DP(S_{k+1}=x)\eta_k(x)\\
      &\hspace{6mm}+\sum_{|x|\leq \alpha \sqrt{k}} \delta_{k,x}\DP(S_{k+1}=x)\eta_k(x).
  \end{align*}
    For the second term of the right hand side, it is easy to check as in Proposition~\ref{neglect B and C} that
    $$\lim_{k\to\infty} k^{d/2} \IE \left( \sum_{|x|\leq \alpha \sqrt{k}} \delta_{k,x}\DP(S_{k+1}=x)\eta_k(x)\right)^2=0.$$

For the first term, denoting by $B(r)$ the closed ball of $\mathbb{R}^d$ of radius $r$, we compute,
    \begin{align}
      &\quad k^{d/2} \left(\sum_{|x|\leq \alpha \sqrt{k}} W_{l_k} \overleftarrow{W}^x_{k+1,l_{k}} \DP(S_{k+1}=x)\eta_k(x)\right)^2 \nonumber\\
      &=k^{d/2} \sum _{x,y\in B(\alpha\sqrt{k})} W_{l_k}^2 \overleftarrow{W}^x_{k+1,l_{k}} \overleftarrow{W}^y_{k+1,l_{k}} \eta_k(x)\eta_k(y)\DP(S_{k+1}=x)\DP(S_{k+1}=y) \nonumber\\
      &= k^{d/2}\underset{|x-y|_1\leq l_k}{\sum _{x,y\in B(\alpha\sqrt{k})}} W_{l_k}^2 \overleftarrow{W}^x_{k+1,l_{k}} \overleftarrow{W}^y_{k+1,l_{k}} \eta_k(x)\eta_k(y) \DP(S_{k+1}=x)\DP(S_{k+1}=y) \nonumber\\
      &\hspace{6mm}+ W_{l_k}^2 k^{d/2}\underset{|x-y|_1>l_k}{\sum _{x,y\in B(\alpha\sqrt{k})}}  \overleftarrow{W}^x_{k+1,l_{k}} \overleftarrow{W}^y_{k+1,l_{k}} \eta_k(x)\eta_k(y)\DP(S_{k+1}=x)\DP(S_{k+1}=y)\nonumber\\
      &=: \mathcal{D}^{(1)}_{k}+W_{l_k}^2 \mathcal{D}^{(2)}_{k}. \label{eq:decFirstTerm}
    \end{align}
    By Cauchy-Schwarz inequality and Theorem \ref{th:StandardLLT},
\[
\left|\mathcal{D}^{(1)}_{k}\right|\leq C k^{-d/2} l_k^d \sum _{x\in B(\alpha\sqrt{k})} W_{l_k}^2 \left(\overleftarrow{W}^x_{k+1,l_{k}} \right)^2 \eta_k(x)^2.
\]
By \eqref{prob-0}, Lemma~\ref{prod of UI} and uniform integrability of $W_{n}^2$ (note that $W_n^2$ converges in $L^1$), we get that as $k\to\infty$,
    \begin{align*}
a_k:&= \sup_{0\leq m< k/2} \sup_{x\in \mathbb{Z}^d}  \IE\left[ W_{m}^2 \left(\overleftarrow{W}^x_{k+1,m} \right)^2 \eta_k(x)^2 \mathbf{1}_{\{ k^{\frac{d-2}{4}}|D_{k+1}|>\eps\}} \right]\\
& \to 0,
    \end{align*}
where, in order to use Lemma~\ref{prod of UI}, we have restricted the supremum to $m<k/2$, so that $W_{m}$ and $\overleftarrow{W}^x_{k+1,m}$ are then  independent from each other, and, by definition, independent of $\eta_k(x)$. Then, we choose and fix a specific $(l_k)_{k}$, which satisfies both $l_k^d\, a_k \to 0$ and the conditions of Proposition~\ref{th:L2LLT} (and hence $l_k<k/2$ for large $k$). Thereby, as $k\to\infty$,
    \begin{align*}
 \IE\left[k^{-d/2}l_k^{d} \sum _{x\in B(\alpha\sqrt{k})} W_{l_k}^2 \left(\overleftarrow{W}^x_{k+1,l_k} \right)^2 \eta_k(x)^2\mathbf{1}_{\{ k^{\frac{d-2}{4}}|D_{k+1}|>\eps\}} \right]\leq C\,l_k^d\, a_k \to 0.
        \end{align*}
As a consequence, we have
    $$\lim_{k\to\infty} \IE\left[\mathcal{D}^{(1)}_{k}\mathbf{1}_{\{ k^{\frac{d-2}{4}}|D_{k+1}|>\eps\}}\right]=0.$$

    Finally, note that
  \begin{align*}
    &\quad \IE \left[\left(\mathcal{D}_k^{(2)}\right)^2\right]\\
    &=  k^{d}\underset{|x-y|_1>l_k}{\sum _{x,y\in B(\alpha\sqrt{k})}}\underset{|z-w|_1>l_k}{\sum _{z,w\in B(\alpha\sqrt{k})}} \IE\left[\prod_{u\in\{x,y,z,w\}}  \overleftarrow{W}^u_{k+1,l_{k}} \eta_k(u)\DP(S_{k+1}=u)\right],
  \end{align*}
where, by independence of $\eta_k(u)$ and $\overleftarrow{W}^u_{k+1,l_{k}}$, each term inside the sum vanishes unless either $x=z$, $y=w$ or $x=w$, $y=z$. Hence, by Theorem \ref{th:StandardLLT}, we obtain that:
    \begin{align*}
      \IE \left[\left(\mathcal{D}_k^{(2)}\right)^2\right]
      &\leq C k^{-d} \underset{|x-y|_1>l_k}{\sum _{x,y\in B(\alpha\sqrt{k})}}  \IE \left[\left(\overleftarrow{W}^x_{k+1,l_{k}}\right)^2 \left(\overleftarrow{W}^y_{k+1,l_{k}}\right)^2 \eta_k(x) ^2 \eta_k(y)^2\right]\\
    &\leq C k^{-d} \sum _{x,y\in B(\alpha\sqrt{k})}  \IE \left[\left(\overleftarrow{W}^0_{k+1,l_{k}}\right)^2 \right]^2 \IE \left[\eta_k(0) ^2 \right]^2\\
    &=O(1),
    \end{align*}
where the second inequality comes from independence of $\overleftarrow{W}^x_{k+1,l_{k}}$, $\overleftarrow{W}^y_{k+1,l_{k}}$ whenever $|x-y| > l_k$. Therefore, $\mathcal{D}_{k}^{(2)}$ is uniformly integrable, so by independence of  $W_{l_k}^2$ and $\mathcal{D}_{k}^{(2)}$ and Lemma~\ref{prod of UI},
\[\lim_{k\to\infty} \IE\left[W_{l_k}^2 \mathcal{D}^{(2)}_{k}\mathbf{1}_{\{ k^{\frac{d-2}{4}}|D_{k+1}|>\eps\}}\right]=0.\]

Putting things together, we have shown \eqref{fixed k Lind}.
    
\end{proof}
\subsection{Proof of Corollary \ref{cor:mainCor}} \label{subsec:proofOfMainCor}
\begin{proof}
  We write
\begin{equation*}
  \log{W_\infty} - \log{W_n}= \log{\left(1+\frac{W_\infty -W_n}{W_n}\right)}.
\end{equation*}
%
By Taylor expansion, there exists a constant $M>0$, such that for all $|x|<1/2$, we have
\begin{equation}\label{difference-log}
  |\log{(1+x)}-x|\leq M{x^2}.
  \end{equation}
Then, we write $X_n=\frac{W_\infty -W_n}{W_n}$, so that by Theorem \ref{th:mainTheorem}, $n^{\frac{d-2}{4}}X_n \cvlaw \sigma G$ and this convergence is mixing. In particular $X_n\cvIP 0$.

By the inequality in \eqref{difference-log}, we have
\begin{equation*}
 \IP\left( n^{\frac{d-2}{4}} |\log(1+X_n)- X_n| > \eps ;~|X_n|< 1/2\right) \leq \IP\left(Mn^{\frac{d-2}{4}}|X_n|^2>\eps\right),
\end{equation*}
which vanishes as $n\to\infty$. Moreover, $\IP(|X_n|\geq 1/2)\to 0$, so that
\[n^{\frac{d-2}{4}}(\log{(1+X_n)}-X_n) \cvIP 0.\]

\begin{lemma}
  Suppose that $Y_n \cvlaw Y$ and $Z_n \cvIP 0$, where $Y$ has a continuous cumulative distribution function. Then
  \begin{equation}\label{converge-YZ}
    Y_n+Z_n \cvlaw Y.
  \end{equation}
  Moreover, if, in addition, the convergence $Y_n \cvlaw Y$ is mixing, the convergence \eqref{converge-YZ} is also mixing.
\end{lemma}
\begin{proof}
  Let us denote by $(\Omega,\mathcal{F},\IP)$ the probability space. Recall that the property that $Y_n \cvlaw Y$ is mixing is equivalent to that for any $x\in\R$ and $B\in\mathcal{F}$ with $\IP(B)>0$,
$$\lim_{n\to\infty}\IP(Y_n\leq x;~B)=\IP(Y\leq x)\IP(B).$$
  We fix $x\in\R$ and $B\in\mathcal{F}$ with $\IP(B)>0$. For any $\eps>0$,
  \begin{align*}
    \limsup_{n\to\infty}\IP(Y_n+Z_n\leq x;~B)&\leq \lim_{n\to\infty} \IP(Y_n\leq x+\eps;~B)+\lim_{n\to\infty}\IP(Z_n<-\eps)\\
    &= \IP(Y\leq x+\eps)\IP(B).
  \end{align*}
  Letting $\eps \downarrow 0$, since $Y$ has a continuous cumulative distribution function, we have
  $$\limsup_{n\to\infty}\IP(Y_n+Z_n\leq x;~B)\leq\IP(Y\leq x)\IP(B).$$
  Conversely, for any $\eps>0$,
  \begin{align*}
    \liminf_{n\to\infty}\IP(Y_n+Z_n\leq x;~B)&\geq \liminf_{n\to\infty}\IP(Y_n+Z_n\leq x;~Z_n\leq \eps;~B)\\
    &\geq \lim_{n\to\infty} \IP(Y_n\leq x-\eps;~B)-\lim_{n\to\infty}\IP(Z_n>\eps)\\
    &\geq \IP(Y\leq x-\eps)\IP(B).
  \end{align*}
  Similarly, we have
   $$\liminf_{n\to\infty}\IP(Y_n+Z_n\leq x;~B)\geq \IP(Y\leq x)\IP(B).$$
  \end{proof}
Using this lemma, we get 
$$ n^{\frac{d-2}{4}}\log{\left(1+\frac{W_\infty -W_n}{W_n}\right)}\cvlaw \sigma G,$$
and this convergence is mixing.
\end{proof}

\section*{Acknowledgments}

The authors would like to thank Francis Comets for giving them the opportunity to meet, as well as for his careful reading and helpful comments. The authors are also grateful to NYU Shanghai, for the hospitality during the stay where the present work was initiated.

\end{document}